\documentclass[12pt,reqno]{amsart}
\usepackage[notref,notcite]{}
\usepackage{amssymb}
\usepackage{amsmath}

\newtheorem{theorem}{Theorem}[section]
\newtheorem{lemma}[theorem]{Lemma}
\newtheorem{proposition}[theorem]{Proposition}
\newtheorem{corollary}[theorem]{Corollary}

\numberwithin{equation}{section}

\def\ep{\varepsilon}

\def\D{\mathbb D}
\def\C{\mathbb C}
\def\R{\mathbb R}
\def\S{\mathbb S}
\def\N{\mathbb N}
\def\pa{\partial}
\def\b{\backslash}

\def\l{\langle}
\def\r{\rangle}
\def\Z{\mathbb Z}
\def\T{{\rm Tr}}
\def\U{\mathcal{U}}
\def\H{\mathcal H}
\def\F{\mathcal F}

\begin{document}

\title[Convexity properties of the difference $\zeta_\Omega-\zeta_\D$]{Convexity properties of the difference over the real axis between the Steklov zeta functions of a smooth planar domain with $2\pi$ perimeter and of the unit disk }
\author{Alexandre Jollivet}
\thanks{
The author is partially supported by the PRC n\textsuperscript{o} 2795 CNRS/RFBR : Probl\`emes inverses et int\'egrabilit\'e.}

\address{Laboratoire de Math\'ematiques Paul Painlev\'e,
CNRS UMR 8524/Universit\'e Lille 1 Sciences et Technologies,
59655 Villeneuve d'Ascq Cedex, France}
\email{alexandre.jollivet@math.univ-lille1.fr}



\begin{abstract}
We consider the zeta function $\zeta_\Omega$ for the Dirichlet-to-Neumann operator of a simply connected planar domain $\Omega$ bounded by a smooth closed curve of perimeter $2\pi$. We prove that $\zeta_\Omega''(0)\ge \zeta_{\D}''(0)$ with equality if and only if $\Omega$ is a disk where $\D$ denotes the closed unit disk. We also provide an elementary proof that for a fixed  real $s$ satisfying $s\le-1$ the estimate $\zeta_\Omega''(s)\ge \zeta_{\D}''(s)$ holds with equality if and only if $\Omega$ is a disk. We then bring examples of domains $\Omega$ close to the unit disk where this estimate fails to be extended to the interval $(0,2)$. Other computations related to previous works are also detailed in the remaining part of the text.
\end{abstract}

\maketitle
\section{Introduction}

Let
$\Omega$
be a simply connected
planar domain bounded by a
$C^\infty$-smooth closed curve
$\partial\Omega$.
The {\it Dirichlet-to-Neumann operator} of the domain
$$
\Lambda_\Omega:C^\infty(\partial\Omega)\rightarrow C^\infty(\partial\Omega)
$$
is defined by
$\Lambda_\Omega f=\left.\frac{\partial u}{\partial\nu}\right|_{\partial\Omega}$,
where
$\nu$
is the outward unit normal to
$\partial\Omega$
and
$u$
is the solution to the Dirichlet problem
$$
\Delta u=0\quad\mbox{\rm in}\quad\Omega,\quad u|_{\partial\Omega}=f.
$$
The Dirichlet-to-Neumann operator is a first order pseudodifferential operator. Moreover, it is a non-negative self-adjoint operator with respect to the $L^2$-product
$$
\l u,v\r=\int\limits_{\partial\Omega}u\bar v\,ds,
$$
where
$ds$
is the Euclidean arc length of the curve
$\partial\Omega$.
In particular, the operator
$\Lambda_\Omega$
has a non-negative discrete eigenvalue spectrum
$$
\mbox{\rm Sp}(\Omega)=\{0=\lambda_0(\Omega)<\lambda_1(\Omega)\leq\lambda_2(\Omega)\leq\dots\},
$$
where each eigenvalue is repeated according to its multiplicity. The spectrum is called the {\it Steklov spectrum} of the domain
$\Omega$.
Steklov eigenvalues depend on the size of
$\Omega$
in the obvious manner:
$\lambda_k(c\Omega)=c^{-1}\lambda_k(\Omega)$
for
$c>0$.
Therefore it suffices to consider domains satisfying the {\it normalization condition}
\begin{equation}
\mbox{Length}(\partial\Omega)=2\pi.
                                     \label{1.1}
\end{equation}

Let
${\mathbb S}=\partial{\mathbb D}=\{e^{i\theta}\}\subset{\mathbb C}$
be the unit circle. The Dirichlet-to-Neumann operator of the unit disk
${\mathbb D}=\{(x,y)\mid x^2+y^2\leq1\}$
will be denoted by
$\Lambda:C^\infty({\mathbb S})\rightarrow C^\infty({\mathbb S})$,
i.e.,
$\Lambda=\Lambda_{\mathbb D}$.
The alternative definition of the operator is given by the formula
$\Lambda e^{in\theta}=|n|e^{in\theta}$
for an integer
$n$.
Then the Steklov eigenvalues of the disk are given by
$$
\lambda_k(\D)=\left\lfloor{k+1\over 2}\right\rfloor, k\in\N,
$$
where
$\lfloor x\rfloor$
stands for the integer part of
$x\in\R$.

Under condition \eqref{1.1}, Steklov eigenvalues of the domain $\Omega$ have the following asymptotics \cite[Theorem~1]{E}:
\begin{equation}
\lambda_k(\Omega)=\lambda_k(\D) +O(k^{-\infty})\quad\mbox{as}\quad k\to\infty,
                                     \label{1.2}
\end{equation}

Due to the asymptotics, the {\it zeta function of the domain}
$\Omega$
$$
\zeta_\Omega(s)=\mbox{\rm Tr}[\Lambda_\Omega^{-s}]=\sum\limits_{k=1}^\infty\big(\lambda_k(\Omega)\big)^{-s}
$$
is well defined for
$\Re s>1$. 
Then
$\zeta_\Omega$
extends to a meromorphic function on
${\mathbb C}$
with the unique simple pole at
$s=1$.  The zeta function $\zeta_\D$ of the unit disk is equal to $2\zeta_R$, where
$\zeta_R(s)=\sum_{n=1}^\infty n^{-s}$
is the classical Riemann zeta function

Moreover, the difference
$\zeta_\Omega(s)-\zeta_\D(s)$
is an entire function \cite{E}. Observe also that
$\zeta_\Omega(s)$
is real for a real
$s$.

The main result of the present paper is the following

\begin{theorem}
\label{thm1}
For a smooth simply connected bounded planar domain
$\Omega$ of perimeter $2\pi$, the inequality
\begin{equation}
\sum_{k=1}^{\infty}\Big[\ln(\lambda_k)^2-\ln\big(\left\lfloor{k+1\over 2}\right\rfloor\big)^2\Big]=(\zeta_\Omega-\zeta_{\D})''(0)\ge 0 \label{M1a}
\end{equation}
holds. Moreover equality in \eqref{M1a} holds if and only if $\Omega$ is a round disk.
\end{theorem}

Inequality \eqref{M1a} is a straighforward consequence of the identity $\zeta_\Omega(0)=\zeta_\D(0)$ and of the estimate $(\zeta_\Omega-\zeta_\D)\ge 0$ on the real axis $\R$ \cite[Theorem 1.1]{JS4}. Equality in \eqref{M1a} trivially holds if $\Omega$ is a round disk (in that case $\zeta_\Omega=\zeta_\D$). Hence the only statement that remains to be proved is the ``only if" part. The proof relies on the same deformation argument we used to prove \cite[Theorem 1.1]{JS4}.

The above result proves the strict convexity of $\zeta_\Omega-\zeta_{\D}$ around $0$ when the planar domain $\Omega$ is not a disk. It is in fact easier to prove convexity of $\zeta_\Omega-\zeta_{\D}$ on $(-\infty,-1]$. We have the following result.

\begin{proposition}
\label{prop1}
Let $\Omega$ be a smooth simply connected bounded domain with $2\pi$ perimeter. Let $s\in(-\infty,-1]$. We have
$$
(\zeta_\Omega-\zeta_\D)''(s)\ge 0,
$$
and there is equality if and only if $\Omega$ is a round disk.
\end{proposition}

Convexity near $+\infty$ is also granted by Weinstock's inequality \cite{We} and we have the following result.

\begin{proposition}
\label{prop2}
Let $\Omega$ be a smooth simply connected bounded domain with $2\pi$ perimeter. Assume that $\Omega$ is not a round disk. Then there exists a positive real $s_\Omega$ so that for any $s\in [s_\Omega,+\infty)$ the inequality 
\begin{equation}
\label{M1c}
(\zeta_\Omega-\zeta_\D)''(s)>0
\end{equation}
holds.
\end{proposition}

It is then questionable whether one can extend the statement to the whole real axis. We exhibit counterexamples in the following Proposition.

\begin{proposition}
\label{prop3}
There exist a smooth simply connected bounded planar domain
$\Omega$ of perimeter $2\pi$ and a real number $s\in (0,2)$ so that
$$
(\zeta_\Omega-\zeta_\D)''(s)< 0.
$$
\end{proposition}

Now, we discuss an alternative approach to the same results which are of a more analytical character.

For a function
$b\in C^\infty({\mathbb S})$,
we write
$b(\theta)$
instead of
$b(e^{i\theta})$
and use the same letter
$b$
for the operator
$b:C^\infty({\mathbb S})\rightarrow C^\infty({\mathbb S})$
of multiplication by the function
$b$.

Given a positive function
$a\in C^\infty({\mathbb S})$,
the operator
$\Lambda_a=a^{1/2}\Lambda a^{1/2}$
has the non-negative discrete eigenvalue spectrum
$$
\mbox{\rm Sp}(\Lambda_a)=\{0=\lambda_0(a)<\lambda_1(a)\leq\lambda_2(a)\leq\dots\}
$$
which is called the {\it Steklov spectrum of the function}
$a$
(or of the operator
$\Lambda_a$).

Two kinds of the Steklov spectrum are related as follows. Given a smooth simply connected planar domain
$\Omega$,
choose a biholomorphism
$\Phi:{\mathbb D}\rightarrow\Omega$
and define the function
$0<a\in C^\infty({\mathbb S})$
by
$a(\theta)=|\Phi'(e^{i\theta})|^{-1}$.
Let
$\phi:{\mathbb S}\rightarrow\partial\Omega$
be the restriction of
$\Phi$
to
${\mathbb S}$.
Then
$\Lambda_a=a^{-1/2}\phi^*\Lambda_\Omega\,\phi^{*-1}a^{1/2}$
and
$\mbox{\rm Sp}(\Lambda_a)=\mbox{\rm Sp}(\Omega)$.
Two latter equalities make sense for an arbitrary positive function
$a\in C^\infty({\mathbb S})$
if we involve multi-sheet domains into our consideration. See \cite[Section 3]{JS} for details. Theorem \ref{thm1} is true for multi-sheet domains as well. The normalization condition \eqref{1.1} is written in terms of the function
$a$
as follows:
\begin{equation}
\frac{1}{2\pi}\int\limits_0^{2\pi}\frac{d\theta}{a(\theta)}=1.
                                     \label{1.4}
\end{equation}

The biholomorphism
$\Phi$
of the previous paragraph is defined up to a conformal transformation of the disk
$\D$,
this provides examples of functions with the same Steklov spectrum. Two functions
$a,b\in C^\infty({\mathbb S})$
are said to be {\it conformally equivalent}, if there exists a conformal or anticonformal transformation
$\Psi$
of the disk
${\mathbb D}$
such that
$b=|d\psi/d\theta|^{-1}a\circ\psi$,
where the function
$\psi(\theta)$
is defined by
$e^{i\psi(\theta)}=\Psi(e^{i\theta})$
($\Psi$
is anticonformal if
$\bar\Psi$
is conformal).
If two positive functions
$a,b\in C^\infty({\mathbb S})$
are conformally equivalent, then
$\mbox{Sp}(a)=\mbox{Sp}(b)$.

Under condition \eqref{1.4}, Steklov eigenvalues
$\lambda_k(a)$
have the same asymptotics \eqref{1.2}.
The {\it zeta function of}
$a$
is defined by
\begin{equation}
\zeta_a(s)=\mbox{\rm Tr}[\Lambda_a^{-s}]=\sum\limits_{k=1}^\infty\big(\lambda_k(a)\big)^{-s}
                                            \label{1.5}
\end{equation}
for
$\Re(s)>1$.
It again extends to a meromorphic function on
${\mathbb C}$
with the unique simple pole at
$s=1$
such that
$\zeta_a(s)-2\zeta_R(s)$
is an entire function. Here the Steklov zeta function $\zeta_{\mathbf 1}$ of the constant function $\mathbf 1$
(= the constant function identically equal to 1) is equal to $\zeta_\D(=2\zeta_R)$. 

The analytical versions of Theorem \ref{thm1}, Propositions \ref{prop1}, \ref{prop2} and \ref{prop3} sound as follows:

\begin{theorem}
\label{thm2}
For a positive function
$a\in C^\infty({\mathbb S})$
satisfying the normalization condition \eqref{1.4}, the inequality
\begin{equation}
(\zeta_a-2\zeta_R)''(0)\ge 0\label{M1}
\end{equation}
holds. Moreover equality in \eqref{M1} holds if and only if $a$ is conformally equivalent to the constant function $\mathbf 1$.
\end{theorem}

\begin{proposition}
\label{prop4}
Let $s\in (-\infty,-1]$.  For a positive function
$a\in C^\infty({\mathbb S})$
satisfying the normalization condition \eqref{1.4}, the inequality
\begin{equation}
(\zeta_a-2\zeta_R)''(s)\ge 0\label{M1b}
\end{equation}
holds. Moreover equality in \eqref{M1b} holds if and only if $a$ is conformally equivalent to the constant function $1$.
\end{proposition}

\begin{proposition}
\label{prop5}
Let $a$ be a positive function $a\in C^\infty({\mathbb S})$ satisfying the normalization condition \eqref{1.4}. Assume that $a$ is not conformally equivalent to $\mathbf 1$. Then there exists a positive real $s_a$ so that for any $s\in [s_a,+\infty)$ the inequality 
\begin{equation}
\label{M1d}
(\zeta_a-2\zeta_R)''(s)> 0
\end{equation}
holds.
\end{proposition}

\begin{proposition}
\label{prop6}
Let $\U$ be an open neighborhood of $1$ in $C^\infty(\S)$.
There exist a smooth positive function $a\in \U$ satisfying the normalization condition \eqref{1.4} and a real number $s\in (0,2)$ so that
$$
(\zeta_a-2\zeta_R)''(s)< 0.
$$
\end{proposition}

Actually Proposition \ref{prop6} is a strengthened version of Proposition \ref{prop3}.

The paper is organized as follows. We prove Propositions \ref{prop4} and \ref{prop5} in Section 2. We prove Theorem \ref{thm2} in Section 3.
We prove Proposition \ref{prop6} in Section 4.
The last Sections 5, 6  and 7 are apart from the convexity questions. We expand the quantities
$\l\ln(\Lambda_a+P_0)\phi_n,\phi_n\r$ in a $C^\infty$ neighborhood of the  constant function $\mathbf 1$. Here
the Hilbert space $L^2(\S)$ is considered with the scalar product
$$
\l u, v\r=\int_{\S}u(\theta)\overline{v(\theta)}d\theta,
$$
and $(\phi_n)_{n\in\Z}$ is the orthonormal basis defined by
\begin{equation}
\phi_n(\theta)={1\over \sqrt{2\pi a(\theta)}}e^{in\int_0^\theta a^{-1}(s)ds},\ \theta\in [0,2\pi),\  n\in \Z,\label{2.1}
\end{equation}
and $P_0$ is the orthogonal projection onto the kernel of $\Lambda_a$.
In particular we prove that the identity $\l\ln(\Lambda_a+P_0)\phi_n,\phi_n\r=\ln(|n|)$, $n\not=0$, does not hold in general, which was the impetus for the deformation argument that leads to \cite[Theorem 1.1]{JS4}, see the concluding remarks given in \cite[Section 7]{JS4}.

\section{Strict convexity on $(-\infty,-1]$ and near $+\infty$: Proof of Propositions \ref{prop4} and \ref{prop5}}

In this Section we first recall some notations and properties and we then prove Propositions \ref{prop4} and \ref{prop5}.

\subsection{Notations, powers and logarithm of operators}
We use the derivative
$$
D=-i\frac{d}{d\theta}:C^\infty({\mathbb S})\rightarrow C^\infty({\mathbb S}).
$$
Operators $D$ and $\Lambda$ have the same one-dimensional null-space consisting of constant functions. 

We define the first order differential operator  $D_a:C^\infty({\mathbb S})\rightarrow C^\infty({\mathbb S})$ by
$$
D_a=a^{1/2}Da^{1/2}.
$$
The orthonormal basis $(\phi_n)_{n\in \Z}$ defined by \eqref{2.1} is an eigenbasis for $D_a$:
\begin{equation}
D_a\phi_n=n\phi_n,\ n\in \Z.\label{2.5}
\end{equation}
We will denote $|D_a|=(D_a^2)^{1\over 2}$. And we denote $P_0$ the orthogonal projection of $L^2(\S)$ onto the one-dimensional space spanned by the function $\phi_0$.

Let $f$ be a function from $(0,+\infty)$ to $\R$ with at most a polynomial growth at $+\infty$: $|f(x)|=O(x^N)$ as $x\to+\infty$ for some integer $N$.
Let $A$  be a positive pseudodifferential operator of order one with a discrete eigenvalue spectrum:
If $\{\psi_k\}_{k\in \N}$ is an orthonormal basis of $L^2(\S)$
consisting of eigenvectors of $A$ with associated eigenvalues $\lambda_k>0$, then
$$
f(A)u=\sum_{k\in\N}f(\lambda_k)\l u,\psi_k \r\psi_k\ \mbox{for}\ u\in C^\infty(\S).
$$
The operator $f(A): C^\infty(\S)\to C^\infty(\S)$ defines a (possibly unbounded) selfadjoint operator in $L^2(\S)$. 
In this paper we consider only the case when $A=\Lambda_a+P_0$ or $A=|D_a|+P_0$ and $f(x)=x^s\ln^m(x)$ for $s\in \R$ and $m=0,1,2$. 

For instance equality \eqref{2.5} implies
\begin{equation}
f(|D_a|+P_0)\phi_n=f\big(\max(|n|,1)\big)\phi_n\quad n\in\Z.
                         \label{2.8}
\end{equation}

When $f$ is convex then we recall that 
\begin{equation}
\l f(A)u,v\r\ge f(\l A u,v\r) \label{2.6}
\end{equation}
for $(u,v)\in C^\infty(\S)^2$ so that $\l u,v\r=1$ and $\l u, \psi_k\r\l\psi_k,v\r\ge 0$ for every $k\in \N$
(see for instance the proof of \cite[Lemma 5.2]{JS4}).

Let $s\in\R$ and $m\in \N$. The difference 
$$
(\Lambda_a+P_0)^{-s}\ln^m(\Lambda_a+P_0)-(|D_a|+P_0)^{-s}\ln^m(|D_a|+P_0)
$$ 
is a smoothing operator and 
\begin{equation}
(-1)^m{d^m (\zeta_a-2\zeta_R)\over ds^m}(s)=\T\big[(\Lambda_a+P_0)^{-s}\ln^m(\Lambda_a+P_0)-(|D_a|+P_0)^{-s}\ln^m(|D_a|+P_0)\big],\label{2.7}
\end{equation}
see \cite[Lemmas 3.4 and 3.5]{JS4} where the operator ``$H(\tau,z)$" is taken at $\tau=0$ and $z=s$.

\subsection{Proof of Proposition \ref{prop4}}
First we use \eqref{2.7} when $m=2$:
$$
(\zeta_a''-2\zeta_R'')(-s)=\T((\Lambda_a+P_0)^s\ln(\Lambda_a+P_0)^2-(|D_a|+P_0)^s\ln(|D_a|+P_0)^2)),
$$
and we expand the trace with respect to the basis $(\phi_n)_{n\in \Z}$ 
\begin{equation}
(\zeta_a-2\zeta_R)''(-s)=\sum_{n\in\Z\b\{0\}}\Big(\l(\Lambda_a+P_0)^s\ln^2(\Lambda_a+P_0)\phi_n,\phi_n\r-|n|^s\ln(|n|)^2\Big).\label{M30}
\end{equation}

Let $n\in \Z\b\{0\}$. By Cauchy-Bunyakovsky-Schwarz inequality we have
\begin{eqnarray}
&&\big(\l(\Lambda_a+P_0)^s\ln(\Lambda_a+P_0)\phi_n,\phi_n\r\big)^2=\big(\l(\Lambda_a+P_0)^{s\over 2}\ln(\Lambda_a+P_0)\phi_n,(\Lambda_a+P_0)^{s\over 2}\phi_n\r\big)^2\nonumber\\
&\le&\l(\Lambda_a+P_0)^s\phi_n,\phi_n\r\l(\Lambda_a+P_0)^s\ln^2(\Lambda_a+P_0)\phi_n,\phi_n\r\label{2.9}
\end{eqnarray}
for $s\in \R$. Now set $s\ge 1$. 
We recall the estimate \cite[Lemmas 2.1 and 2.4]{JS2}
\begin{equation}
\l(\Lambda_a+P_0)^s\phi_n,\phi_n\r\ge |n|^s\ge 1.\ \label{2.10}
\end{equation}
We divide both sides of the inequality \eqref{2.9} by $\l(\Lambda_a+P_0)^s\phi_n,\phi_n\r$ and we obtain
\begin{eqnarray}
\l(\Lambda_a+P_0)^s\ln^2(\Lambda_a+P_0)\phi_n,\phi_n\r
&\ge& {\big(\l(\Lambda_a+P_0)^s\ln(\Lambda_a+P_0)\phi_n,\phi_n\r\big)^2\over \l(\Lambda_a+P_0)^s\phi_n,\phi_n\r}\nonumber\\
&=&{1\over s^2}{\big(\l f(\Lambda_a+P_0)^s)\phi_n,\phi_n\r\big)^2\over \l(\Lambda_a+P_0)^s\phi_n,\phi_n\r}\label{2.11}
\end{eqnarray}
where $f$ is the convex function $f(x)=x\ln(x)$, $x>0$. 
We used the identity $\ln(\Lambda_a+P_0)=s^{-1}\ln((\Lambda_a+P_0)^s)$. Then we use \eqref{2.6}: 
\begin{equation}
\l f\big((\Lambda_a+P_0)^s\big)\phi_n,\phi_n\r\ge f(\l(\Lambda_a+P_0)^s\phi_n,\phi_n\r)\ge 0.\label{2.12}
\end{equation}
The nonnegativity in \eqref{2.12} follows from \eqref{2.10}. Then we combine \eqref{2.11} and \eqref{2.12}
and we obtain
\begin{eqnarray}
\l(\Lambda_a+P_0)^s\ln^2(\Lambda_a+P_0)\phi_n,\phi_n\r
&\ge&{1\over s^2}{f\big(\l(\Lambda_a+P_0)^s\phi_n,\phi_n\r)\big)^2\over \l(\Lambda_a+P_0)^s\phi_n,\phi_n\r}\nonumber\\
&=&{1\over s^2}\l(\Lambda_a+P_0)^s\phi_n,\phi_n\r\ln \Big(\l(\Lambda_a+P_0)^s\phi_n,\phi_n\r\Big)^2\nonumber\\
&\ge &|n|^s\ln(|n|)^2.\label{M31}
\end{eqnarray}
We used \eqref{2.10} at the last line.

Inequality \eqref{M1b} follows from \eqref{M30} and \eqref{M31}.
Equality in \eqref{M1b} implies that each summand in \eqref{M30} is zero:
$\l(\Lambda_a+P_0)^s\ln^2(\Lambda_a+P_0)\phi_n,\phi_n\r=|n|^s\ln(|n|)^2$ for $n\in \Z\b\{0\}$. In particular it implies equalities in \eqref{M31}. Therefore $\l(\Lambda_a+P_0)^s\phi_n,\phi_n\r=|n|^s$ for $n\in \Z\b\{0\}$. The identity for $n=1$ is enough to conclude that $a$ is conformally equivalent to $\mathbf 1$ \cite[Lemma 2.5]{JS2}.\hfill $\Box$

\subsection{Proof of Proposition \ref{prop5}}
Assume that $a$ is not conformally equivalent to the constant function $\mathbf 1$.
Weinstock's inequality \cite{We} tells us that
$$
\lambda_1(a)<1.
$$
And by definition
$$
\zeta_a(s)-2\zeta_R(s)=\sum_{k=1}^\infty\Big[\lambda_k(a)^{-s}-\big(\lfloor{k+1\over 2}\rfloor\big)^{-s}\Big],
$$
\begin{eqnarray*}
\zeta_a''(s)-2\zeta_R''(s)&=&\lambda_1(a)^{-s}\ln(\lambda_1(a))^2+\lambda_2(a)^{-s}\ln(\lambda_2(a))^2\\
&&+\sum_{k=3}^\infty\Big[\lambda_k(a)^{-s}\ln(\lambda_k(a))^2-\big(\lfloor{k+1\over 2}\rfloor\big)^{-s}\ln\big(\lfloor{k+1\over 2}\rfloor\big)^2\Big].
\end{eqnarray*}

Hence the leading order as $s\to +\infty$ in the above sum is $\lambda_1(a)^{-s}\ln(\lambda_1(a))^2$.
The asymptotics makes obvious the existence of $s_a\in[0,+\infty)$ so that 
\begin{equation}
\zeta_a''(s)-2\zeta_R''(s)>0, s\ge s_a.\label{M20}
\end{equation}
{}\hfill $\Box$

\section{The second derivative $\zeta_\Omega''$ at $0$: Proof of Theorem \ref{thm2}}

The proof of Theorem \ref{thm2} relies on the same deformation argument used to prove \cite[Theorem 1.1]{JS4}.
We start this Section by recalling some definition of a variation of a function $a$. We give a proof of Theorem \ref{thm2} at the end.

\subsection{Deformation of a function $a$ and the Hilbert transform $\H$}

Let $l=0$ or $l=\infty$ and let $\ep>0$. A real function $\alpha\in C^l\big([0,\varepsilon), C^\infty({\mathbb S})\big)$
is called a $C^l$-{\it deformation} (or $C^l$-{\it variation}) of a positive function
$a\in C^\infty({\mathbb S})$ when it satisfies the 3 conditions: $\alpha(0,\theta)=a(\theta)$;
For any $\tau\in [0,\ep)$ the function
$\alpha_\tau=\alpha(\tau,\cdot )\in C^\infty({\mathbb S})$
is positive and it satisfies the normalization condition 
\begin{equation}
\int_{\S}\alpha_\tau^{-1}(\theta)d\theta=2\pi.\label{3.1}
\end{equation}

The  entire function $\zeta_{\alpha_\tau}$ has the following smoothness along the deformation $\alpha$ \cite[Lemma 3.5]{JS4}
\begin{equation}
\zeta_{\alpha_\tau}-2\zeta_R\in C^l([0,\ep)_\tau,\F(\C)).\label{3.2}
\end{equation}
Here $\F(\C)$ denotes the space of entire functions on the complex plane.

The {\it Hilbert transform}
$\H$
is the linear operator on
$L^2(\S)$
defined by
$$
\H({\mathbf 1})=0,\quad \H e^{in\theta}={\rm sgn}(n)e^{in\theta}\ \textrm{for an integer}\ n\neq0.
$$
We will use the identities
\begin{equation}
D=\H\Lambda=\Lambda\H,\ D_a=\Lambda_a a^{-1/2}\H a^{1/2}.\label{3.3}
\end{equation}

\subsection{Preliminary Lemma}

\begin{lemma}
\label{lem_snd_der}
Let $a\in C^\infty(\S)$ be positive and  satisfy the normalization condition \eqref{1.4}. Then
\begin{equation}
\T(\ln(\Lambda_a+P_0)(\Lambda_a+P_0)^{-1}(\Lambda_a^2-D_a^2))\ge 0\label{M6}
\end{equation}
with equality if and only if $a$ is conformally equivalent to the constant function $\mathbf 1$.
\end{lemma}
The operator inside the trace in \eqref{M6} is trace class. Indeed it is the product of the bounded operator $\ln(\Lambda_a+P_0)(\Lambda_a+P_0)^{-1}$ and of the smoothing operator $\Lambda_a^2-D_a^2$ (see Section 2.1).

\begin{proof}[Proof of Lemma \ref{lem_snd_der}]
We expand the trace with respect to the orthonormal basis $(\phi_n)$:
\begin{eqnarray}
&&\T(\ln(\Lambda_a+P_0)(\Lambda_a+P_0)^{-1}(\Lambda_a^2-D_a^2))\nonumber\\
&=&\sum_{n\in\Z\b\{0\}}\big( \l \Lambda_a \ln(\Lambda_a+P_0)\phi_n,\phi_n\r-n^2\l\ln(\Lambda_a+P_0)(\Lambda_a+P_0)^{-1}\phi_n,\phi_n\r\big).\label{M3}
\end{eqnarray}
We used the identity $(\Lambda_a+P_0)^{-1}\Lambda_a^2=\Lambda_a$ and we used \eqref{2.5}. We prove that the summands are nonnegative.

Let $n\in \N\b\{0\}$. First we use \eqref{2.12} for $s=1$:
\begin{equation}
\l \Lambda_a \ln(\Lambda_a+P_0)\phi_n,\phi_n\r\ge \l \Lambda_a\phi_n,\phi_n\r\ln( \l \Lambda_a\phi_n,\phi_n\r).\label{M4}
\end{equation}
In addition we use \eqref{2.5} and \eqref{3.3} and the identity $(\Lambda_a+P_0)^{-1}\Lambda_a=I-P_0$ where $I$ is the identity operator and we obtain
\begin{eqnarray}
n\l\ln(\Lambda_a+P_0)(\Lambda_a+P_0)^{-1}\phi_n,\phi_n\r&=&\l\ln(\Lambda_a+P_0)(\Lambda_a+P_0)^{-1}D_a\phi_n,\phi_n\r\nonumber\\
&=&\l \ln(\Lambda_a+P_0)a^{-1/2}\H a^{1/2}\phi_n,\phi_n\r.\label{3.20a}
\end{eqnarray}

Let 
$$
\delta_n=n^{-1}\l\Lambda_a\phi_n,\phi_n\r,\ u=\delta_n^{-1}a^{-1/2}\H a^{1/2}\phi_n,\ v=\phi_n.
$$
We apply \eqref{2.6} when $f(x)=-\ln(x)$ and $A=\Lambda_a+P_0$  (see details for the sign of $\l u,\psi_k\r\l \psi_k,v\r$ in \cite[Section 5, part 5]{JS4}) 
and we obtain
\begin{equation}
\l \ln(\Lambda_a+P_0)\delta_n^{-1}a^{-1/2}\H a^{1/2}\phi_n,\phi_n\r
\le \ln\big(\l \Lambda_a \delta_n^{-1}a^{-1/2}\H a^{1/2}\phi_n,\phi_n\r\big)=\ln\big(\delta_n^{-1}n\big).\label{3.20b}
\end{equation}
We combine \eqref{3.20a} and \eqref{3.20b} and we obtain
\begin{eqnarray}
&&n^2\l \ln(\Lambda_a+P_0)(\Lambda_a+P_0)^{-1}\phi_n,\phi_n\r\le n\delta_n\ln\big(\delta_n^{-1}n\big)=\l\Lambda_a\phi_n,\phi_n\r\ln\big({n^2\over \l\Lambda_a\phi_n,\phi_n\r}\big)\nonumber\\
&\le &\l\Lambda_a\phi_n,\phi_n\r\ln(n)\le  \l \Lambda_a\phi_n,\phi_n\r\ln( \l \Lambda_a\phi_n,\phi_n\r).\label{M5}
\end{eqnarray}
We used the growth of the logarithm and we used the estimates
$$
\l\Lambda_a\phi_n,\phi_n\r\ge n,\ \delta_n\ge 1,
$$
see \eqref{2.10} for $s=1$. We combine \eqref{M4} and \eqref{M5}  and we obtain
$$
\l \Lambda_a \ln(\Lambda_a+P_0)\phi_n,\phi_n\r\ge n^2\l\ln(\Lambda_a+P_0)(\Lambda_a+P_0)^{-1}\phi_n,\phi_n\r,\ n\in \N\b\{0\}.
$$
We obtain the same estimates for negative integers $n$ by complex conjugation invariance. Then we use again \eqref{M3} and we obtain \eqref{M6}.
 
Now equality in \eqref{M6} means equalities in the last line of \eqref{M5}. Hence
$$
\l\Lambda_a\phi_n,\phi_n\r= n,\ n\in\N\b\{0\}.
$$
The identity for $n=1$ is enough to conclude that $a$ is conformally equivalent to $\mathbf 1$ \cite[Lemma 2.5]{JS2}.
\end{proof}

\subsection{Proof of Theorem \ref{thm2}}
The inequality
$$
(\zeta_a-2\zeta_R)(s)\ge 0,\ s\in \R,
$$
holds by \cite[Theorem 1.1]{JS4}. Since $(\zeta_a-2\zeta_R)(0)=0$ we obtain the inequality \eqref{M1}. In addition $\zeta_a=2\zeta_R$ when $a$ is conformally equivalent to the constant function  $\mathbf 1$. 

Hence we only have to prove that $(\zeta_a-2\zeta_R)''(0)=0$ implies that $a$ is conformally equivalent to a constant function.

Consider the deformation $\alpha\in C^\infty([0,\infty)_\tau, C^\infty(\S))$ introduced in \cite[Theorem 1.3]{JS4}. 
It satisfies the evolution equation
\begin{equation}
{\pa\alpha_\tau\over \pa\tau}=-\alpha_\tau\Lambda\alpha_\tau+\H\alpha_\tau D\alpha_\tau, \tau\ge 0,\label{3.30}
\end{equation}
with initial condition $\alpha_0=a$, and by \cite[Theorem 4.1]{JS4}
\begin{equation}
{\pa\zeta_{\alpha_\tau}\over \pa\tau}(s)=s\T((\Lambda_{\alpha_\tau}+P_{0,\tau})^{-s-1}(\Lambda_{\alpha_\tau}^2-D_{\alpha_\tau}^2)),\label{M10}
\end{equation}
for $s\in \R$ and $\tau\in[0,\infty)$. Here $P_{0,\tau}$ is the orthogonal projection of $L^2(\S)$ onto the one-dimensional space spanned by the function $(2\pi\alpha_\tau)^{-1/2}$.

In addition $\alpha_\tau\to \mathbf{1}$ as $\tau \to\infty$ in $C^\infty$-topology.

Let $s\in \R$. Set $N=|s|+1$.
The operator $(\Lambda_{\alpha_\tau}+P_{0,\tau})^N(\Lambda_{\alpha_\tau}^2-D_{\alpha_\tau}^2)$ is a smoothing operator, see Section 2.1, while $(\Lambda_{\alpha_\tau}+P_{0,\tau})^{-\sigma-1-N}$ is a family of bounded operators in $L^2(\S)$ that is smooth with respect to $\sigma$ in a neighborhood of $s$. Hence we can intertwin the trace on the right hand side of \eqref{M10} and any derivative with respect to the $s$-variable.
We derive \eqref{M10} with respect to $s$ and we denote $'$ or ${d\over ds}$ the derivative with respect to the real variable $s$ and we have
\begin{eqnarray*}
{\pa\zeta_{\alpha_\tau}'\over \pa\tau}(s)
&=&\T((\Lambda_{\alpha_\tau}+P_{0,\tau})^{-s-1}(\Lambda_{\alpha_\tau}^2-D_{\alpha_\tau}^2))\\
&&+s\T({d \over ds}(\Lambda_{\alpha_\tau}+P_{0,\tau})^{-s}(\Lambda_{\alpha_\tau}+P_{0,\tau})^{-1}(\Lambda_{\alpha_\tau}^2-D_{\alpha_\tau}^2))\\
&=&\T((\Lambda_{\alpha_\tau}+P_{0,\tau})^{-s-1}(\Lambda_{\alpha_\tau}^2-D_{\alpha_\tau}^2))\\
&&-s\T(\ln(\Lambda_{\alpha_\tau}+P_{0,\tau})(\Lambda_{\alpha_\tau}+P_{0,\tau})^{-s-1}(\Lambda_{\alpha_\tau}^2-D_{\alpha_\tau}^2))
\end{eqnarray*}
We derive once more in $s$ and we obtain
\begin{eqnarray*}
{\pa\zeta_{\alpha_\tau}''\over \pa\tau}(s)&=&-2\T(\ln(\Lambda_{\alpha_\tau}+P_{0,\tau})(\Lambda_{\alpha_\tau}+P_{0,\tau})^{-s-1}(\Lambda_{\alpha_\tau}^2-D_{\alpha_\tau}^2))\\
&&+s\T(\ln(\Lambda_{\alpha_\tau}+P_{0,\tau})^2(\Lambda_{\alpha_\tau}+P_{0,\tau})^{-s-1}(\Lambda_{\alpha_\tau}^2-D_{\alpha_\tau}^2).
\end{eqnarray*}
Therefore 
\begin{equation}
{\pa\zeta_{\alpha_\tau}''\over \pa\tau}(0)=-2\T(\ln(\Lambda_{\alpha_\tau}+P_{0,\tau})(\Lambda_{\alpha_\tau}+P_{0,\tau})^{-1}(\Lambda_{\alpha_\tau}^2-D_{\alpha_\tau}^2)).
\end{equation}

We apply Lemma \ref{lem_snd_der} to obtain that 
\begin{equation}
\zeta_{\alpha_\tau}''(0)
\textrm{ is nonincreasing in }\tau. \label{M11}
\end{equation}

Moreover since $\alpha_\tau\to 1$ as $\tau\to \infty$ in $C^\infty$-topology, we obtain that 
\begin{equation}
\zeta_{\alpha_\tau}''(0)\to 2\zeta_R''(0),\textrm{ as }\tau \to \infty.\label{M12}
\end{equation}
Indeed we consider the continuous path $\beta\in C([0,\infty)_\ep,C^\infty(\S))$ defined by 
$$
\beta_0=\mathbf{1},\ \beta_\ep=\alpha_{1\over \ep}\textrm{ for }\ep>0.
$$
Then \eqref{3.2} yields 
$$
\zeta_{\beta_\ep}\in C([0,\infty)_\ep,C^\infty(\C\b\{1\}))\textrm{ and }{d^j\zeta_{\beta_\ep}\over ds^j}(0)\to 2{d^j\zeta_R\over ds^j}(0)
\textrm{ as }\ep\to 0^+
$$
for any $j\in \N$. Hence we proved statement \eqref{M12}.

Now assume that $\zeta_a''(0)=2\zeta_R''(0)$. Then we obtain by \eqref{M11} and \eqref{M12} $(\zeta_{\alpha_\tau}-2\zeta_R)''(0)= 0$ for any $\tau$ and 
$$
\T(\ln(\Lambda_{\alpha_\tau}+P_{0,\tau})(\Lambda_{\alpha_\tau}+P_{0,\tau})^{-1}(\Lambda_{\alpha_\tau}^2-D_{\alpha_\tau}^2))=-{1\over 2}{\pa\zeta_{\alpha_\tau}''\over \pa\tau}(0)=0.
$$
Therefore we apply again Lemma \ref{lem_snd_der} and we obtain that $\alpha_\tau$ is conformally equivalent to a constant valued function for any $\tau$. In particular, $a$ is conformally equivalent to $\mathbf 1$.\hfill $\Box$

\section{The difference $\zeta_a-2\zeta_R$ may not be convex everywhere on the real axis: Proof of Proposition \ref{prop6}}

We recall the following result \cite[Proposition 3.8]{JS4}.
\begin{proposition}[see \cite{JS4}]
Let
$\alpha_\tau$
be a
$C^\infty$-variation of the function
$a={\mathbf 1}$. Then, for every
$z\in\C$,
\begin{equation}
\left.{\pa \big(\zeta_{\alpha_\tau}(z)\big)\over \pa \tau}\right|_{\tau=0}=0,
                                  \label{3.54}
\end{equation}
\begin{equation}
\left.{\pa^2 \big(\zeta_{\alpha_\tau}(z)\big)\over \pa \tau^2}\right|_{\tau=0}=
4z\sum_{(n,p)\in \N^2\atop p>0,\ n>0,\ p\not=n}{n^{-z}-p^{-z}\over p^2-n^2}\,pn\,|\hat\beta_{p+n}|^2+2z^2\sum_{n>0}n^{-z}\,|\hat\beta_{2n}|^2,
                                  \label{3.55}
\end{equation}
where
$\beta(\theta)=\left.\frac{\partial\alpha_\tau(\theta)}{\partial\tau}\right|_{\tau=0}$ (and $\alpha_0=\mathbf{ 1}$).
\end{proposition}

The proof of Proposition \ref{prop6} relies on the analysis of the right hand side of \eqref{3.55}. 
From now on we consider only 
$C^\infty$-variation $\alpha_\tau$ of the function
$a={\mathbf 1}$ so that 
\begin{equation}
\beta(e^{i\theta})=2\cos((2r+1)\theta),\ \theta\in \R,\label{M51}
\end{equation}
for some large integer $r$. Take for instance the smooth variation
\begin{equation}
\alpha_\tau(e^{i\theta})=\big(1-2\tau \cos((2r+1)\theta)\big)^{-1},\ \tau\in (-1/2,1/2),\ \theta\in \R.\label{M52}
\end{equation}

The right hand side of \eqref{3.55} becomes
$$
\left.{\pa^2 \big(\zeta_{\alpha_\tau}(s)\big)\over \pa \tau^2}\right|_{\tau=0}=
-4s\sum_{(n,p)\in \N^2\atop p>0,\ n>0,\ p+n=2r+1}{p^{-s}-n^{-s}\over p^2-n^2}\,pn,
$$
for a real $s$ (we used $\hat \beta_{2r+1}=1$).
Hence 
\begin{eqnarray*}
\left.{\pa^2 \big(\zeta_{\alpha_\tau}(s)\big)\over \pa \tau^2}\right|_{\tau=0}= -8s\sum_{n>0,\ p>n,\ p+n=2r+1} {p^{-s}-n^{-s}\over p^2-n^2}pn.
\end{eqnarray*}
We derive with respect to $s$:
\begin{eqnarray}
\left.{\pa^2 \big(\zeta_{\alpha_\tau}'(s)\big)\over \pa \tau^2}\right|_{\tau=0}
&=&{8\over 2r+1}\sum_{p=r+1}^{2r}{p(2r+1-p)\over 2p-2r-1}\label{M53}\\
&&\times (p^{-s}(-1+s\ln(p))-(2r+1-p)^{-s}(-1+s\ln(2r+1-p))).\nonumber
\end{eqnarray}
(We substitued $n$ by $2r+1-p$.)

Let us make an asymptotic analysis as $r\to \infty$, $0<s<2$. 
\begin{eqnarray}
&&\sum_{p=r+1}^{2r}{p(2r+1-p)\over 2p-2r-1}(p^{-s}(-1+s\ln(p))-(2r+1-p)^{-s}(-1+s\ln(2r+1-p)))\nonumber\\
&=&(2r+1)^{1-s}\sum_{p=r+1}^{2r}{{p\over 2r+1}\big(1-{p\over 2r+1}\big)\over 2{p\over 2r+1}-1}\Big(\big({p\over 2r+1}\big)^{-s}\big(-1+s\ln\big({p\over 2r+1}\big)\big)\nonumber\\
&&-\big(1-{p\over 2r+1}\big)^{-s}\big(-1+s\ln\big(1-{p\over 2r+1}\big)\Big)\nonumber\\
&&+s(2r+1)^{1-s}\ln(2r+1)\sum_{p=r+1}^{2r}{{p\over 2r+1}\big(1-{p\over 2r+1}\big)\over 2{p\over 2r+1}-1}\Big(\big({p\over 2r+1}\big)^{-s}-\big(1-{p\over 2r+1}\big)^{-s}\Big).\nonumber
\end{eqnarray}
Therefore 
\begin{eqnarray}
&&\sum_{p=r+1}^{2r}{p(2r+1-p)\over 2p-2r-1}(p^{-s}(-1+s\ln(p))-(2r+1-p)^{-s}(-1+s\ln(2r+1-p)))\nonumber\\
&=&(2r+1)^{2-s}\Big[\int_{1/2}^1{x(1-x)\over 2x-1}(x^{-s}(-1+s\ln(x))\nonumber\\
&&-(1-x)^{-s}(-1+s\ln(1-x)))dx+o(1)\Big]\nonumber\\
&&+s(2r+1)^{2-s}\ln(2r+1)\Big[\int_{1/2}^1{x(1-x)\over 2x-1}(x^{-s}-(1-x)^{-s})+o(1)\Big]\label{M54}
\end{eqnarray}
as $r\to +\infty$.
We used the following elementary statement for the singularity near $x=1$, see for instance \cite[Section 2.12.7]{DR}:
For a continuous function $\eta\in C((0,1),\R)$ so that  $\eta(x)=O(x^{-\rho})$, $\eta(1-x)=O(x^{-\rho})$ as $x\to 0^+$, $0<\rho<1$, then 
$$
\int_0^1\eta=\lim_{N\to \infty}{1\over N}\sum_{i=1}^{N-1}\eta({i\over N}).
$$

Then the leading order is given by $(2r+1)^{2-s}\ln(2r+1)$ as $r\to +\infty$ and the coefficient in front of it is 
\begin{equation}
s\int_{1/2}^1{x(1-x)\over 2x-1}(x^{-s}-(1-x)^{-s})dx<0\label{M55}
\end{equation}
since $x>1-x$ and $x^{-s}<(1-x)^{-s}$ for $x\in (1/2,1)$ and $s\in (0,2)$.

Now we combine \eqref{M53}, \eqref{M54} and \eqref{M55} and we obtain that at fixed $s\in (0,2)$ there exists  a large integer $r_s$ so that
\begin{equation}
\left.{\pa^2 \big(\zeta_{\alpha_\tau}'(s)\big)\over \pa \tau^2}\right|_{\tau=0}<0\label{M56}
\end{equation}
for any integer $r\ge r_s$ (let us remind that the path $\alpha_\tau$ is defined by the integer $r$).

From \eqref{3.54} it also follows that
\begin{equation}
\left.{\pa \big(\zeta_{\alpha_\tau}'(s)\big)\over \pa \tau}\right|_{\tau=0}=0.\label{M57}                                  
\end{equation}
At $\tau=0$, $\zeta_{\alpha_\tau}'(s)=2\zeta_R'(s)$.

Therefore we make a Taylor expansion of $\zeta_{\alpha_\tau}'(s)$ with respect to $\tau$ in a neighborhood of $0$ and we obtain that there exists $\tau_s$
$$
\zeta_{\alpha_\tau}'(s)<2\zeta_R'(s)
$$
for any integer $r\ge r_s$ and any $\tau\in (0,\tau_s)$.

Now we are ready to conclude the proof of Proposition \ref{prop6}. Take $\U$ any neighborhood of $\mathbf 1$ in $C^\infty(\S)$. Let $s\in (0,2)$. Then define the integer $r_s$ and the positive real number $\tau_s$ as above and choose $\tau\in (0,\tau_s)$ small enough so that 
$$
\alpha_\tau\in \U.
$$ 
Such an $\alpha_\tau$ plays the role of $a$ in the statement of Proposition \ref{prop6}.
Indeed $\zeta_{\alpha_\tau}'(s)<2\zeta_R'(s)$. Since $\zeta_{\alpha_\tau}'(0)=2\zeta_R'(0)$ (see \cite{EW,JS4}) the function $\zeta_{\alpha_\tau}-2\zeta_R$ is not convex in $(0,2)$.\hfill $\Box$\\

We can go beyond the interval $(0,2)$. Now let $s>2$
\begin{eqnarray*}
\left.{\pa^2 \big(\zeta_{\alpha_\tau}'(s)\big)\over \pa \tau^2}\right|_{\tau=0}
&=&{8\over 2r+1}\sum_{p=r+1}^{2r}{p(2r+1-p)\over 2p-2r-1}\\
&&\times(p^{-s}(-1+s\ln(p))-(2r+1-p)^{-s}(-1+s\ln(2r+1-p)))\\
&=&{8\over 2r+1}\sum_{p=1}^{2r}{(2r+1-p)\over 2p-2r-1}p^{1-s}(-1+s\ln(p))
\end{eqnarray*}
(We make a change of variable ``$p$''$=2r+1-p$ at the last line.)
Then
\begin{eqnarray*}
&&\sum_{p=1}^{2r}{(2r+1-p)\over 2p-2r-1}p^{1-s}(-1+s\ln(p))\\
&=&\sum_{p=1}^{2r}p^{1-s}(1-s\ln(p))-\sum_{p=1}^{2r}{p^{2-s}(1-s\ln(p))\over 2p-2r-1}\\
&=&\sum_{p=1}^{2r}p^{1-s}(1-s\ln(p))-\sum_{p=\lfloor \sqrt{r}\rfloor+1}^{2r}{p^{2-s}(1-s\ln(p))\over 2p-2r-1}
-\sum_{p=1}^{\lfloor \sqrt{r}\rfloor}{p^{2-s}(1-s\ln(p))\over 2p-2r-1}.
\end{eqnarray*}
We conclude using the elementary facts:
Note that
$$
\sum_{p=1}^{2r} p^{1-s}(1-s\ln(p))\to \zeta_R(s-1)+s\zeta_R'(s-1)\textrm{ as }r\to\infty,
$$
(here the series is absolutely convergent) and  
\begin{eqnarray*}
&&|\sum_{p=\lfloor \sqrt{r}\rfloor+1}^{2r}{p^{-s+2}\over 2p-2r-1}(1-s\ln(p))|\le r^{1-s/2}(1+|s|\ln(2r))\sum_{p=\lfloor \sqrt{r}\rfloor+1}^{2r}{1\over |2p-2r-1|}\\
&\le &2r^{1-s/2}(1+|s|\ln(2r))(1+\int_1^{2r}{dt\over t})=2r^{1-s/2}(1+|s|\ln(2r))(1+\ln(2r))\to 0,
\end{eqnarray*} 
\begin{eqnarray}
\big|\sum_{p=1}^{\lfloor \sqrt{r}\rfloor}{p^{2-s}(1-s\ln(p))\over 2p-2r-1}\big|&\le& (2r+1-2\sqrt{r})^{-1}\sum_{p=1}^{\lfloor \sqrt{r}\rfloor}p^{2-s}(1+|s|\ln(p))\\
&\le&{\sqrt{r}\over 2r+1-2\sqrt{r}}\sum_{p=1}^{\lfloor \sqrt{r}\rfloor}p^{1-s}(1+|s|\ln(p))\to 0,
\end{eqnarray}
as $r\to +\infty$.
Hence we finally obtain that
\begin{equation}
{2r+1\over 8}\left.{\pa^2 \big(\zeta_{\alpha_\tau}'(s)\big)\over \pa \tau^2}\right|_{\tau=0}\to \zeta_R(s-1)+s\zeta_R'(s-1)\textrm{ as }r\to\infty.\label{M58}
\end{equation}

We recall the formula \cite[Chapter 2, Section 2.1]{Tit}:
$$
\zeta_R(z)={z\over z-1}-z\int_1^{\infty}{u-\lfloor u\rfloor\over u^{1+z}}du.
$$
There is a single pole at $z=1$ and there exists $s_0\in (2,\infty)$  so that 
$$
\zeta_R(s-1)+s\zeta_R'(s-1)<0,\ s\in(2,s_0).
$$
Note that $s_0<+\infty$ since $\zeta_R(z)\to 1$ and $(z+1)\zeta_R'(z)\to 0$ as $\Re z\to +\infty$.

Let $s\in (2,s_0)$ and for $r$ large enough we again obtain \eqref{M56}. Repeating the proof of Proposition \ref{prop6} we obtained the following result.

\begin{proposition}
\label{prop7}
Let $\U$ be an open neighborhood of $1$ in $C^\infty(\S)$ and let $s\in (2,s_0)$.
There exists a smooth positive function $a\in \U$  so that
$$
(\zeta_a'-2\zeta_R')(s)< 0.
$$
\end{proposition}

\section{Addenda: Comments}
We proved in \cite{JS4} that 
$$
(\zeta_a-2\zeta_R)(s)\ge 0
$$
for any $s\in \R$ and any smooth positive function $a\in C^\infty(\S)$ satisfying the normalizing condition \eqref{1.4}. We used a path of deformaton $\alpha_\tau$ (see Section 2). The estimate $(\zeta_a-2\zeta_R)(s)\ge 0$ was actually proved in an easier way when  $|s|\ge 1$ in \cite{JS2}: The proof relied on the estimates 
$$
\l(\Lambda_a+P_0)^s\phi_n,\phi_n\r\ge |n|^s\textrm{ for }(n,s)\in \Z\times (-\infty,-1]\cup [1,+\infty).
$$ 
In this Section we focus on the loss of these estimates for $s$ in a neighborhood of $0$.

\subsection{Main result}
Let $\alpha_\tau\in C^\infty(\S)$,  $\tau\in [0,\ep)$, be a $C^\infty$-variation of $\mathbf 1$.  We denote $P_{0,\tau}$ the orthogonal projection of $L^2(\S)$ onto the one-dimensional space spanned by the function $(2\pi\alpha_\tau)^{-1/2}$.
We recall that the operator $\ln(\Lambda_{\alpha_\tau}+P_{0,\tau})-\ln(|D_{\alpha_\tau}|+P_{0,\tau})$ is a smoothing operator at each $\tau$, see Section 2.1. Actually it defines a family of smoothing operators which is smooth with respect to $\tau$:
\begin{equation}
\ln(\Lambda_{\alpha_\tau}+P_{0,\tau})-\ln(|D_{\alpha_\tau}|+P_{0,\tau})\in C^\infty([0,\ep)_\tau,\mathcal{L}(H^l(\S),H^{l'}(\S))),\label{5.1}
\end{equation}
for any nonnegative real $l,l'$ where $\mathcal{L}(H^l(\S),H^{l'}(\S))$ denotes the Banach space of bounded operators from the Sobolev space $H^l(\S)$ of order $l$ to the one of order $l'$, see \cite[Lemma 3.4]{JS4} (``$\ln(\Lambda_{\alpha_\tau}+P_{0,\tau})-\ln(|D_{\alpha_\tau}|+P_{0,\tau})=-{\pa H\over \pa z}(\tau,0)$'' there).

For each $\tau$ we consider the orthornomal basis $(\phi_{m,\tau})_{m\in \Z}$ of eigenvectors of  $|D_{\alpha_\tau}|$:
$$
\phi_{m,\tau}(\theta)={1\over \sqrt{2\pi\alpha_\tau(\theta)}}e^{im\int_0^\theta \alpha_\tau^{-1}(s)ds},\
m\in\Z,\ \theta\in \R.
$$
We remind that $\ln(|D_{\alpha_\tau}|+P_{0,\tau})\phi_m=\ln|m|\phi_m$. By \eqref{5.1}
$$
\l\ln(\Lambda_{\alpha_\tau}+P_{0,\tau})\phi_{m,\tau},\phi_{m,\tau}\r\in C^\infty([0,\ep)_\tau,\R).
$$
for $m\in\Z\b\{0\}$.

We expand $\l\ln(\Lambda_{\alpha_\tau}+P_{0,\tau})\phi_{m,\tau},\phi_{m,\tau}\r$ at order $2$ in a neighborhood of $\tau=0$.

\begin{theorem}
\label{thm3}
Let $m$ be a positive integer. We have 
\begin{equation}
{d\over d\tau}\l\ln(\Lambda_{\alpha_\tau}+P_{0,\tau})\phi_{m,\tau},\phi_{m,\tau}\r_{|\tau=0}=0,\label{K60a}
\end{equation}
\begin{eqnarray}
{1\over 4}{d^2\over d\tau^2}\l\ln(\Lambda_{\alpha_\tau}+P_{0,\tau})\phi_{m,\tau},\phi_{m,\tau}\r_{|\tau=0}
=\sum_{p>0,\ p\not=m}\gamma(p,m)|\hat \beta_{m+p}|^2,\label{K60b}
\end{eqnarray}
where 
$$
\gamma(p,m)={pm\Big(2\ln(p/m)mp
+(p+m)(m-p)\Big)\over (m-p)^2(m+p)^2}
$$
for $(p,m)\in(\N\b\{0\})^2,\ p\not=m$.
Therefore the following Taylor expansion at $0$ holds
\begin{equation}
\l\ln(\Lambda_{\alpha_\tau}+P_{0,\tau})\phi_{m,\tau},\phi_{m,\tau}\r
=\ln m+2\tau^2\sum_{p>0,\ p\not=m}\gamma(p,m)|\hat \beta_{m+p}|^2+o(\tau^2),\ \tau\to 0^+.\label{K60c}
\end{equation}
\end{theorem}

\subsection{Examples of smooth variations $\alpha_\tau$}
Let $m$ be a positive integer. The sign of $m$ is not relevant since $\l\ln(\Lambda_{\alpha_\tau}+P_{0,\tau})\phi_{m,\tau},\phi_{m,\tau}\r$ has the same value when $m$ is replaced by its opposite.

When 
$$
\alpha_\tau=\big(1-2\tau\cos(r\theta)\big)^{-1}
$$
for a positive integer $r> m$ then \eqref{K60b} gives 
\begin{eqnarray*}
{1\over 4}{d^2\over d\tau^2}\l\ln(\Lambda_{\alpha_\tau}+P_{0,\tau})\phi_{m,\tau},\phi_{m,\tau}\r_{|\tau=0}
&=&\gamma(r-m,m)\\
&=&{(r-m)m\Big(2\ln({r-m\over m})m(r-m)
+r(2m-r)\Big)\over (2m-r)^2r^2}
\end{eqnarray*}
and 
$$
{1\over 4}{d^2\over d\tau^2}\l\ln(\Lambda_{\alpha_\tau}+P_{0,\tau})\phi_{r-m,\tau},\phi_{r-m,\tau}\r_{|\tau=0}
=\gamma(m,r-m)=-\gamma(r-m,m).
$$

We consider the asymptotic regime when $r\to \infty$. In that case we have 
$$
\gamma(r-m,m)=-r^{-1}m+o(r^{-1})\textrm{ as }r\to +\infty.
$$
Therefore for $r$ large enough with respect to $m$ 
$$
{d^2\over d\tau^2}\l\ln(\Lambda_{\alpha_\tau}+P_{0,\tau})\phi_{m,\tau},\phi_{m,\tau}\r_{|\tau=0}<0,
$$
and 
$$
{d^2\over d\tau^2}\l\ln(\Lambda_{\alpha_\tau}+P_{0,\tau})\phi_{r-m,\tau},\phi_{r-m,\tau}\r_{|\tau=0}>0,
$$
Then the expansion \eqref{K60c} implies that for small enough positive $\tau$
$$
\l\ln(\Lambda_{\alpha_\tau}+P_{0,\tau})\phi_{m,\tau},\phi_{m,\tau}\r<\ln m,\ \l\ln(\Lambda_{\alpha_\tau}+P_{0,\tau})\phi_{r-m,\tau},\phi_{r-m,\tau}\r>\ln(r-m).
$$
Hence for small enough positive real $s$ and small enough positive $\tau$
$$
\l (\Lambda_{\alpha_\tau}+P_{0,\tau})^s\phi_{m,\tau},\phi_{m,\tau}\r<m^s,\ \l (\Lambda_{\alpha_\tau}+P_{0,\tau})^{-s}\phi_{r-m,\tau},\phi_{r-m,\tau}\r<(r-m)^{-s}.
$$

\subsection{Consequence}
We collect the example of deformations of the previous subsection and we obtain the following result. In the statement given below $\tilde m$ is the integer $r-m$ of the previous subsection.

\begin{corollary}
\label{cor1}
Let $\U$ be an open neighborhood of $1$ in $C^\infty(\S)$ and let $m$ be a positive integer.
There exists a smooth positive function $a\in \U$ that satisfies the normalizing condition \eqref{1.4} and there exist $s_m\in (0,1)$ and a positive integer $\tilde m$ so that
$$
\l(\Lambda_a+P_0)^s\phi_m,\phi_m\r <m^s,\ \l(\Lambda_a+P_0)^{-s}\phi_{\tilde m},\phi_{\tilde m}\r <\tilde m^{-s}
$$
for any $s\in (0,s_m)$.
\end{corollary}

Section 6 is devoted to  preliminary Lemmas for the proof of Theorem \ref{thm3}. We conclude the proof of Theorem \ref{thm3} in Section 7. Both Sections 6 and 7 consist mainly in elementary and technical computations.

\section{Preliminary Lemmas for the proof of Theorem \ref{thm3}}

Let $\alpha_\tau$, $\tau\in[0,\ep)$,  be a $C^\infty$-variation of $\mathbf 1$ for some $\ep>0$.
We recall the definition of the smooth functions $\phi_{m,\tau}\in C^\infty(\S)$ and the definition of the operators $P_{0,\tau},\ \Lambda_{\alpha_\tau}:C^\infty(\S)\to  C^\infty(\S)$ that depend smoothly on $\tau$:
$$
\phi_{m,\tau}(e^{i\theta})={1\over \sqrt{2\pi}}\alpha_\tau^{-1/2}e^{im\int_0^\theta \alpha_\tau^{-1}(e^{is})ds},\
m\in\Z,
$$
and 
$$
P_{0,\tau}=\alpha_\tau^{-{1/2}}P_{e_0}\alpha_\tau^{-1/2},\
\Lambda_{\alpha_\tau}=\alpha_\tau^{1/2}\Lambda\alpha_\tau^{1/2}.
$$
Here $P_{e_0}$ is the orthogonal projection onto the line spanned by $e_0=(2\pi)^{-1/2}$.

Straightforward computations yield the following Lemma.
\begin{lemma}
\label{lem_basic}
Let $m\in \Z$. We have
\begin{eqnarray}
\big({d\over d\tau}\phi_{m,\tau}\big)_{|\tau=0}&=&(2\pi)^{-1/2}\big({1\over 2}f+im\int_0^\theta f)e^{im\theta},\label{K70a}\\
\big({d^2\over d\tau^2}\phi_{m,\tau}\big)_{|\tau=0}&=&(2\pi)^{-1/2}\Big[\big({1\over 2}f+im\int_0^\theta f\big)^2
+{1\over 2}F-{1\over 2}f^2+im\int_0^\theta F\Big]e^{im\theta},\label{K70b}
\end{eqnarray}
where 
$
f(\theta)={\pa \alpha_\tau^{-1}\over \pa\tau}_{|\tau=0}(\theta)=-\beta(\theta)
$
and 
$
F(\theta)={\pa^2 \alpha_\tau^{-1}\over \pa\tau^2}_{|\tau=0}(\theta).
$
In addition
\begin{equation}
(\Lambda_{\alpha_\tau})_{|\tau=0}=\Lambda,\ 
(P_{0,\tau})_{|\tau=0}=P_{e_0},\label{K70f}
\end{equation}
\begin{equation}
\big({d\over d\tau}\Lambda_{\alpha_\tau}\big)_{|\tau=0}=-{1\over 2}(f \Lambda+\Lambda f),\ 
\big({d\over d\tau}P_{0,\tau}\big)_{|\tau=0}={1\over 2}(fP_{e_0}+P_{e_0}f), \label{K70c}
\end{equation}
\begin{equation}
\big({d^2\over d\tau^2}\Lambda_{\alpha_\tau}\big)_{\ep=0}=-{1\over 2}\Big(F
\Lambda+\Lambda F\Big)
+{1\over 4}(3f^2\Lambda+2f\Lambda f+3\Lambda f^2),\label{K70d}
\end{equation}
\begin{equation}
\big({d^2\over d\tau^2}P_{0,\tau}\big)_{|\tau=0}={1\over 2}\Big( F
P_{e_0}+P_{e_0}F\Big)+{1\over 4}(-f^2 P_{e_0}+2f P_{e_0} f-P_{e_0} f^2).\label{K70e}
\end{equation}
\end{lemma}

Next we consider the operator $\ln(\Lambda_{\alpha_\tau}+P_{0,\tau}):C^\infty(\S)\to C^\infty(\S)$ that depends smoothly on $\tau$, see Section 5.1. Smoothness in $\tau$ is understood here as $\ln(\Lambda_{\alpha_\tau}+P_{0,\tau})\phi\in C^\infty([0,\ep)_\tau\times \S)$ for any $\phi\in C^\infty(\S)$. We compute the first and second derivatives of $\ln(\Lambda_{\alpha_\tau}+P_{0,\tau})$ at $\tau=0$. 
We introduce the function $\rho:\Z^2\times(0,+\infty)\to [0,+\infty)$ defined by 
$$
\rho(n,p,s)=\int_0^s\max(|n|,1)^t\max(|p|,1)^{s-t}dt,\ (n,p,s)\in\Z^2\times(0,+\infty).
$$
In other words
$$
\rho(n,p,s)=s\max(|p|,1)^s\textrm{ when }\max(|n|,1)=\max(|p|,1)
$$
and 
\begin{equation}
\rho(n,p,s)={\max(|n|,1)^s-\max(|p|,1)^s\over \ln(\max(|n|,1))-\ln(\max(|p|,1))}\textrm{ otherwise.}\label{K90a}
\end{equation}
We also introduce the function $h:\N\b\{0\}\times\Z\to [0,+\infty)$ defined by 
$$
h(m,p)={1\over 2 m^2}\textrm{ when }\max(|p|,1)=m,
$$
and 
\begin{equation}
h(m,p)={\ln(m)-\ln(\max(|p|,1))\over (m-\max(|p|,1))^2}-{1\over m(m-\max(|p|,1))}\textrm{ otherwise.}\label{K90b}
\end{equation}

\begin{lemma}
\label{lem_basic2}
Let $(n,p)\in\Z^2$ and let $m$ be a positive integer. We have 
\begin{equation}
{1\over 2\pi}\l {\pa\over \pa\tau}\ln(\Lambda_{\alpha_\tau}+P_{0,\tau})_{|\tau=0} e^{ip\theta},e^{in\theta}\r
={1\over 2}\rho(n,p,1)^{-1}(-|n|-|p|+\delta_n+\delta_p)\hat f_{n-p},\label{K46a}
\end{equation}
\begin{eqnarray}
{1\over 2\pi}\l\big[{\pa^2\over \pa\tau^2}\ln(\Lambda_{\alpha_\tau}+P_{0,\tau})\big]_{|\tau =0}e^{im\theta},e^{im\theta}\r
&=&m^{-1}\Big[
{3m\over 4\pi} \int_{\S}f^2+{1\over 2}\sum_{p\in\Z}|p| |\hat f_{m-p}|^2+{1\over 2}|\hat f_m|^2\Big]\nonumber\\
&&-{1\over 2}\sum_{p\in \Z}h(m,p)(|m|+|p|-\delta_p)^2|\hat f_{m-p}|^2.\label{K46b}
\end{eqnarray}
Here $\delta_n$ is the Kronecker symbol: $\delta_n=1$ when $n=0$ and $\delta_n=0$ otherwise.
\end{lemma}

\begin{proof}[Proof of Lemma \ref{lem_basic2}]
We recall that the operator $(\Lambda_{\alpha_\tau}+P_{0,\tau})^t-(|D_{\alpha_\tau}|+P_{0,\tau})^t$ is a smoothing operator at each $\tau$ and $t\in \R$ \cite[Lemma 3.4]{JS4}. Now we take into account that 
$$
(|D_{\alpha_\tau}|+P_{0,\tau})^t\phi =\sum_{l\in\Z}(\max(|l|,1))^t\l\phi,\phi_{l,\tau}\r\phi_{l,\tau}\in C^\infty([0,\ep)_\tau\times\R_t\times \S,\C),
$$
and it follows that $(\Lambda_{\alpha_\tau}+P_{0,\tau})^t\phi\in C^\infty([0,\ep)_\tau\times\R_t\times \S,\C)$
for any $\phi\in C^\infty(\S)$. In this Section any operator is  considered as a linear operator from $C^\infty(\S)$ to $C^\infty(\S)$ and smoothness with respect to either the $t$-variable or $\tau$-variable refers to the pointwise smoothness, i.e. when the operator is applied to any $\phi\in C^\infty(\S)$.

Moreover
$$
{\pa\over \pa t}(\Lambda_{\alpha_\tau}+P_{0,\tau})^t=(\Lambda_{\alpha_\tau}+P_{0,\tau})^t\ln(\Lambda_{\alpha_\tau}+P_{0,\tau}),
$$
and 
$$
{\pa\over \pa t}\big({\pa\over \pa \tau}(\Lambda_{\alpha_\tau}+P_{0,\tau})^t\big)=\big({\pa\over \pa \tau}(\Lambda_{\alpha_\tau}+P_{0,\tau})^t\big)\ln(\Lambda_{\alpha_\tau}+P_{0,\tau})
+(\Lambda_{\alpha_\tau}+P_{0,\tau})^t{\pa\ln(\Lambda_{\alpha_\tau}+P_{0,\tau})\over \pa \tau}.
$$
Hence 
$$
{\pa\over \pa t}\Big[\big({\pa\over \pa \tau}(\Lambda_{\alpha_\tau}+P_{0,\tau})^t\big)(\Lambda_{\alpha_\tau}+P_{0,\tau})^{1-t}\Big]=
(\Lambda_{\alpha_\tau}+P_{0,\tau})^t{\pa\ln(\Lambda_{\alpha_\tau}+P_{0,\tau})\over \pa \tau}(\Lambda_{\alpha_\tau}+P_{0,\tau})^{1-t}.
$$
We integrate over $t\in[0,1]$
\begin{eqnarray}
\int_0^1(\Lambda_{\alpha_\tau}+P_{0,\tau})^t{\pa\over \pa \tau}\ln(\Lambda_{\alpha_\tau}+P_{0,\tau})(\Lambda_{\alpha_\tau}+P_{0,\tau})^{1-t} dt
={\pa\over \pa \tau}(\Lambda_{\alpha_\tau}+P_{0,\tau}),\label{K80}
\end{eqnarray}
for $\tau\in [0,\ep)$. 

Note that 
$$
(\Lambda_{\alpha_\tau}+P_{0,\tau})^s_{|\tau=0}e^{il\theta}=\max(1,|l|)^se^{il\theta}
$$ 
for any $l\in \Z$ and any $s\in \R$ by \eqref{K70f}. Therefore we set $\tau=0$ on the left hand side of \eqref{K80} and we apply it to the vector $(2\pi)^{-1/2}e^{ip\theta}$ and  we take the scalar product with $(2\pi)^{-1/2}e^{in\theta}$ and we obtain by linearity of the integral over $t$ and selfadjointness of $(\Lambda_{\alpha_\tau}+P_{0,\tau})^t$
\begin{eqnarray}
&&\hskip-0.5cm(2\pi)^{-1}\l\int_0^1\big[(\Lambda_{\alpha_\tau}+P_{0,\tau})^t\big]_{|\tau=0}\big[{\pa\over \pa \tau}\ln(\Lambda_{\alpha_\tau}+P_{0,\tau})\big]_{|\tau=0}\big[(\Lambda_{\alpha_\tau}+P_{0,\tau})^{1-t}\big]_{|\tau=0} dt e^{ip\theta},e^{in\theta}\r\nonumber\\
&=&(2\pi)^{-1}\int_0^1\max(1,|p|)^{1-t}\l\big[(\Lambda_{\alpha_\tau}+P_{0,\tau})^t\big]_{|\tau=0}\big[{\pa\over \pa \tau}\ln(\Lambda_{\alpha_\tau}+P_{0,\tau})\big]_{|\tau=0}e^{ip\theta},e^{in\theta}\r\nonumber\\
&=&(2\pi)^{-1}\int_0^1\max(1,|p|)^{1-t}\l \big[{\pa\over \pa \tau}\ln(\Lambda_{\alpha_\tau}+P_{0,\tau})\big]_{|\tau=0}e^{ip\theta},(\Lambda_{\alpha_\tau}+P_{0,\tau})^te^{in\theta}\r dt\nonumber\\
&=&(2\pi)^{-1}\int_0^1 (\max(1,|n|)^t\max(1,|p|))^{1-t} dt
\l{\pa\over \pa \tau}\ln(\Lambda_{\alpha_\tau}+P_{0,\tau})_{|\tau=0}e^{ip\theta},e^{in\theta} \r.\label{K81a}
\end{eqnarray}
Now set $\tau=0$ on the right hand side of \eqref{K80} and apply it to the vector $(2\pi)^{-1/2}e^{ip\theta}$ and  take the scalar product with $(2\pi)^{-1/2}e^{in\theta}$ and use \eqref{K70c} to obtain
\begin{equation}
(2\pi)^{-1}\l{\pa\over \pa \tau}(\Lambda_{\alpha_\tau}+P_{0,\tau})_{|\tau=0}e^{ip\theta},e^{in\theta}\r=-{1\over 2}(|n|+|p|)\hat f_{n-p}
+{1\over 2}\delta_p\hat f_n+{1\over 2}\delta_n\hat f_{-p}.\label{K81b}
\end{equation}
We equate \eqref{K81a} and \eqref{K81b} and this gives \eqref{K46a}.

We prove \eqref{K46b}. First we replace the operator $\Lambda_{\alpha_\tau}+P_{0,\tau}$ by the operator $(\Lambda_{\alpha_\tau}+P_{0,\tau})^t$ in \eqref{K80} and we obtain that 
\begin{eqnarray*}
{\pa\over \pa\tau}(\Lambda_{\alpha_\tau}+P_{0,\tau})^t&=&\int_0^t(\Lambda_{\alpha_\tau}+P_{0,\tau})^{r-t}{\pa\over \pa \tau}\ln(\Lambda_{\alpha_\tau}+P_{0,\tau})(\Lambda_{\alpha_\tau}+P_{0,\tau})^r dr.
\end{eqnarray*}
We repeat the same reasoning as in \eqref{K81a} and we obtain
\begin{eqnarray}
\l {\pa\over \pa\tau}(\Lambda_{\alpha_\tau}+P_{0,\tau})^t_{|\tau=0} e^{ip\theta},e^{in\theta}\r
&=&\rho(n,p,t)\l {\pa\over \pa \tau}\ln(\Lambda_{\alpha_\tau}+P_{0,\tau}) e^{ip\theta},e^{in\theta}\r\nonumber\\
&=&{1\over 2}{\rho(n,p,t)\over\rho(n,p,1)}(-|n|-|p|+\delta_n+\delta_p)\hat f_{n-p}.\label{K82}
\end{eqnarray}

Now we derive \eqref{K80} with respect to $\tau$ and we obtain
\begin{eqnarray}
&&\int_0^1(\Lambda_{\alpha_\tau}+P_{0,\tau})^t{\pa^2\over \pa \tau^2}\ln(\Lambda_{\alpha_\tau}+P_{0,\tau})(\Lambda_{\alpha_\tau}+P_{0,\tau})^{1-t} dt\nonumber\\
&=&{\pa^2\over \pa \tau^2}(\Lambda_{\alpha_\tau}+P_{0,\tau})-\int_0^1{\pa\over \pa\tau}(\Lambda_{\alpha_\tau}+P_{0,\tau})^t{\pa\over \pa \tau}\ln(\Lambda_{\alpha_\tau}+P_{0,\tau})(\Lambda_{\alpha_\tau}+P_{0,\tau})^{1-t} dt\nonumber\\
&&-\int_0^1(\Lambda_{\alpha_\tau}+P_{0,\tau})^t{\pa\over \pa \tau}\ln(\Lambda_{\alpha_\tau}+P_{0,\tau}){\pa\over \pa\tau}(\Lambda_{\alpha_\tau}+P_{0,\tau})^{1-t} dt.\label{K83}
\end{eqnarray}

Set $\tau=0$ on the left hand side of \eqref{K83} and apply it to the vector $(2\pi)^{-1/2}e^{im\theta}$ and take the scalar product with $e^{im\theta}$:
\begin{eqnarray}
&&\hskip -0.5cm(2\pi)^{-1}\l \int_0^1\big[(\Lambda_{\alpha_\tau}+P_{0,\tau})^t\big]_{|\tau=0}\big[{\pa^2\over \pa \tau^2}\ln(\Lambda_{\alpha_\tau}+P_{0,\tau})\big]_{|\tau=0}\big[(\Lambda_{\alpha_\tau}+P_{0,\tau})^{1-t}\big]_{|\tau=0} dt e^{im\theta},e^{im\theta} \r\nonumber\\
&=&m \l \big[{\pa^2\over \pa \tau^2}\ln(\Lambda_{\alpha_\tau}+P_{0,\tau})\big]_{|\tau=0}e^{im\theta},e^{im\theta} \r.\label{K84}
\end{eqnarray}

Then we consider the first term on the right hand side of \eqref{K83}. We use \eqref{K70d} and \eqref{K70e} and we use the identity $\l F e^{im\theta},e^{im\theta}\r=\int_{\S}F=\big[{\pa^2\over \pa\tau^2}\int_{\S}\alpha_\tau^{-1}\big]_{|\tau=0}=0$ by the normalizing condition \eqref{1.4}, and we obtain
\begin{equation}
(2\pi)^{-1}\l {\pa^2\over \pa \tau^2}(\Lambda_{\alpha_\tau}+P_{0,\tau})_{|\tau=0}e^{im\theta},e^{im\theta}\r
={1\over 4}({3m\over \pi} \int_{\S}f^2+2\sum_{p}|p| |\hat f_{m-p}|^2)+{1\over 2}|\hat f_m|^2.\label{K85}
\end{equation}
(We also used that $P_{e_0}e^{im\theta}=0$ for the positive integer $m$.)

Now we consider the second and third terms on the right hand side of \eqref{K83}.
Set $\tau=0$ in the second term and use \eqref{K46a} and \eqref{K82}
\begin{eqnarray}
&&\hskip-0.5cm(2\pi)^{-1}\l\int_0^1\big[{\pa\over \pa\tau}(\Lambda_{\alpha_\tau}+P_{0,\tau})^t\big]_{|\tau=0}\big[{\pa\over \pa \tau}\ln(\Lambda_{\alpha_\tau}+P_{0,\tau})\big]_{|\tau=0}(\Lambda+P_0)^{1-t} dt e^{im\theta},e^{im\theta}\r\nonumber\\
&=&\hskip-0.2cm(2\pi)^{-2}\sum_{p\in \Z}\int_0^1\!\!\!\! m^{1-t}\l \big[{\pa\over \pa\tau}(\Lambda_{\alpha_\tau}+P_{0,\tau})^t\big]_{|\tau=0}e^{ip\theta},e^{im\theta}\r
\l \big[{\pa\over \pa \tau}\ln(\Lambda_{\alpha_\tau}+P_{0,\tau})\big]_{|\tau=0} e^{im\theta} , e^{ip\theta}\r dt\nonumber\\
&=&{m\over 4}\sum_{p\in \Z}\rho^{-2}(m,p,1)(m+|p|-\delta_p)^2|\hat f_{m-p}|^2
\int_0^1\rho(m,p,t)m^{-t} dt\nonumber\\
&=&{m\over 4}\sum_{p\in\Z}{\tilde \rho(m,p)\over \rho^2(m,p,1)}(m+|p|-\delta_p)^2|\hat f_{m-p}|^2.\label{K86}
\end{eqnarray}
Here we introduced the function $\tilde \rho:\N\b\{0\}\times\Z\to [0,+\infty)$ defined by 
$$
\tilde \rho(m,p)=\int_0^1\rho(m,p,s)m^{-s}ds,\ m\in\N\b\{0\},\ p\in\Z.
$$
From \eqref{K90a} it follows that
$$
\tilde \rho(m,p)={1\over 2}\textrm{ when }\max(|p|,1)=m,
$$
and 
$$
\tilde \rho(m,p)={1\over \ln(m)-\ln(\max(|p|,1))}+{\max(|p|,1)-m\over m(\ln(m)-\ln(\max(|p|,1)))^2}\textrm{ otherwise.}
$$
From \eqref{K90b} it follows that
$$
{\tilde \rho(m,p)\over \rho^2(m,p,1)}=h(m,p),\ p\in\Z.
$$

We similarly deal with the third term and we obtain 
\begin{eqnarray}
&&\hskip-1cm(2\pi)^{-1}\l\int_0^1\big[(\Lambda_{\alpha_\tau}+P_{0,\tau})^t\big]_{|\tau=0}\big[{\pa\over \pa \tau}\ln(\Lambda_{\alpha_\tau}+P_{0,\tau}) \big]_{|\tau=0}\big[{\pa\over \pa\tau}(\Lambda_{\alpha_\tau}+P_{0,\tau})^{1-t}\big]_{|\tau=0} dt e^{im\theta},e^{im\theta}\r\nonumber\\
&=&{m\over 4}\sum_{p\in \Z}h(m,p)(|m|+|p|-\delta_p)^2|\hat f_{m-p}|^2.
\label{K87}
\end{eqnarray}
We collect \eqref{K83}--\eqref{K87} and we use \eqref{K90a} and \eqref{K90b} and we obtain \eqref{K46b}.

\end{proof}

\section{Proof of Theorem \ref{thm3}}

First we write
\begin{eqnarray*}
{d\over d\tau}\l\ln(\Lambda_{\alpha_\tau}+P_{0,\tau})\phi_{m,\tau},\phi_{m,\tau}\r
&=&\l\big({d\over d\tau}\ln(\Lambda_{\alpha_\tau}+P_{0,\tau})\big)\phi_{m,\tau},\phi_{m,\tau}\r\\
&&+2\Re\l\ln(\Lambda_{\alpha_\tau}+P_{0,\tau}){d\over d\tau}\phi_{m,\tau},\phi_{m,\tau}\r.
\end{eqnarray*}
And we set $\tau=0$, and we use \eqref{K70a} and \eqref{K46a} and we use the identities $\phi_{m,0}=(2\pi)^{-1/2}e^{im\theta}$, $2\pi\hat f_0= \int_0^{2\pi}f=0$ and we obtain \eqref{K60a}. (The imaginary term $im\int_0^\theta f$ on the right hand side of \eqref{K46a} is disregarded when we take the real part.)

Now we prove \eqref{K60b}. We derive once more in $\tau$ the above formula and we obtain
\begin{eqnarray}
&&{d^2\over d\tau^2}\l\ln(\Lambda_{\alpha_\tau}+P_{0,\tau})\phi_{m,\tau},\phi_{m,\tau}\r
\nonumber\\
&=&\l\big({d^2\over d\tau^2}\ln(\Lambda_{\alpha_\tau}+P_{0,\tau})\big)\phi_{m,\tau},\phi_{m,\tau}\r
+4\Re\l{d\over d\tau}\ln(\Lambda_{\alpha_\tau}+P_{0,\tau}){d\over d\tau}\phi_{m,\tau},\phi_{m,\tau}\r\label{K42}\\
&&+2\Re\l\ln(\Lambda_{\alpha_\tau}+P_{0,\tau}){d^2\over d\tau^2}\phi_{m,\tau},\phi_{m,\tau}\r+2\Re\l\ln(\Lambda_{\alpha_\tau}+P_{0,\tau}){d\over d\tau}\phi_{m,\tau},{d\over d\tau}\phi_{m,\tau}\r\nonumber
\end{eqnarray}

Let us compute the three last terms of the right hand side of \eqref{K42} at $\tau=0$. We use \eqref{K70a} for the fourth term and we obtain
\begin{eqnarray}
\big[\l\ln(\Lambda_{\alpha_\tau}+P_{0,\tau}){d\over d\tau}\phi_{m,\tau},{d\over d\tau}\phi_{m,\tau}\r\big]_{|\tau=0}
&=&(2\pi)^{-1}\sum_{p\in\Z}\ln(\max(|p|,1))|\l e^{ip\theta},{d\over d\tau}\phi_{m,\tau}\r|^2_{\tau=0}\nonumber\\
&=&\sum_{p\in\Z\b\{0\}}\ln(|p|)\Big|{1\over 2}\hat f_{p-m}+im\widehat{\big(\int_0^\theta f\big)}_{p-m}\Big|^2.\label{K43a}
\end{eqnarray}
We use \eqref{K70b} for the third term and we obtain
\begin{eqnarray}
&&\big[\Re\l\ln(\Lambda_{\alpha_\tau}+P_{0,\tau}){d^2\over d\tau^2}\phi_{m,\tau},\phi_{m,\tau}\r\big]_{|\tau=0}
=(2\pi)^{-1/2}\ln(|m|)\Re\l\big({d^2\over d\tau^2}\phi_{m,\tau}\big)_{|\tau=0},e^{im\theta}\r\nonumber\\
&=&(2\pi)^{-1}\ln(|m|)\Re\int_0^{2\pi}\Big[\big({1\over 2}f+im\int_0^\theta f\big)^2
+{1\over 2}F-{1\over 2}f^2+im\int_0^\theta F\Big]d\theta\nonumber\\
&=&(2\pi)^{-1}\ln(|m|)\int_0^{2\pi}\Big[-{1\over 4}f^2-m^2\big(\int_0^\theta f\big)^2
\Big]d\theta.\label{K43b}
\end{eqnarray}
(We used that $\int_{\S}f=\int_{\S}F=0$). We use \eqref{K70a} and \eqref{K46a} for the second term and we obtain
\begin{eqnarray}
&&\Re\l{d\over d\tau}\ln(\Lambda_{\alpha_\tau}+P_{0,\tau}){d\over d\tau}\phi_{m,\tau},\phi_{m,\tau}\r_{|\tau=0}\nonumber\\
&=&(2\pi)^{-1}\Re\l\big({d\over d\tau}\ln(\Lambda_{\alpha_\tau}+P_{0,\tau})\big)_{\tau=0}({1\over 2}f+im\int_0^\theta f)e^{im\theta}, e^{im\theta}\r\nonumber\\
&=&(4\pi)^{-1}\Re\sum_{p\in \Z}{(-m-|p|+\delta_p)\over \rho(m,p,1)}\hat f_{m-p}\l({1\over 2}f+im\int_0^\theta f)e^{im\theta},e^{ip\theta}\r\nonumber\\
&=&{1\over 2}\Re\sum_{p\in \Z}{(-m-|p|+\delta_p)\over \rho(m,p,1)}\big({1\over 2}|\hat f_{m-p}|^2+im\hat f_{m-p}\widehat{\big(\int_0^\theta f\big)}_{p-m}\big).\label{K43c}
\end{eqnarray}

The first term on the right hand side of \eqref{K42} is given by \eqref{K46b} at $\tau=0$.

We collect \eqref{K42}, \eqref{K43a}, \eqref{K43b}, \eqref{K43c} and \eqref{K46b} and we obtain
\begin{eqnarray}
&&\Big[{d^2\over d\tau^2}\l\ln(\Lambda_{\alpha_\tau}+P_{0,\tau})\phi_{m,\tau},\phi_{m,\tau}\r\Big]_{|\tau=0}\nonumber\\
&&\hskip -0.3cm =m^{-1}\Big[
{3m\over 4\pi} \int_{\S}f^2+{1\over 2}\sum_{p}|p| |\hat f_{m-p}|^2+{1\over 2}|\hat f_m|^2\Big]
-{1\over 2}\sum_{p\in \Z}h(m,p)(m+|p|-\delta_p)^2|\hat f_{m-p}|^2\nonumber\\
&&\hskip -0.3cm +2\sum_{p\in\Z\b\{0\}}\ln(|p|)\Big|{1\over 2}\hat f_{p-m}+im\widehat{\big(\int_0^\theta f\big)}_{p-m}\Big|^2
+\pi^{-1}\ln(m)\int_0^{2\pi}\Big[-{1\over 4}f^2-m^2(\int_0^\theta f)^2\Big]d\theta\nonumber\\
&&\hskip -0.3cm +2\Re\sum_{p\in\Z}{(-m-|p|+\delta_p)\over \rho(m,p,1)}\big({1\over 2}|\hat f_{m-p}|^2+im\hat f_{m-p}\widehat{\big(\int_0^\theta f\big)}_{p-m}\big).\label{K44}
\end{eqnarray}
From now on every computations intend to simplify the above formula.
We introduce the function $G$
$$
G(\theta)=\int_0^\theta f,\ \hat f_k=ik\hat G_k,\ k\in \Z.
$$
Then \eqref{K44} becomes
  \begin{eqnarray*}
&&\Big[{d^2\over d\tau^2}\l\ln(\Lambda_{\alpha_\tau}+P_{0,\tau})\phi_{m,\tau},\phi_{m,\tau}\r\Big]_{|\tau=0}\\
&=&m^{-1}\Big[
{3m\over 4\pi} \int_{\S}f^2+{1\over 2}\sum_{p}|p| (m-p)^2|\hat G_{m-p}|^2)+{1\over 2}m^2|\hat G_m|^2\Big]\\
&&-{1\over 2}\sum_{p\in \Z}h(m,p)(m+|p|-\delta_p)^2(m-p)^2|\hat G_{m-p}|^2\\
&&+{1\over 2}\sum_{p\in\Z\b\{0\}}(p+m)^2\ln(|p|)|\hat G_{p-m}|^2
+\pi^{-1}\ln(m)\int_0^{2\pi}\Big[-{1\over 4}f^2-m^2G^2\Big]d\theta\\
&&+\sum_{p\in\Z}{m+|p|-\delta_p\over \rho(m,p,1)}(m-p)(m+p)|\hat G_{m-p}|^2.
\end{eqnarray*}
Next we change $p$ in $-p$ and we use the identities $\int_S f^2=2\pi\sum_{p\in\Z}(m+p)^2|\hat G_{m+p}|^2$ and $\int_S G^2=2\pi\sum_{p\in\Z}|\hat G_{m+p}|^2$ and we obtain
\begin{eqnarray}
&&\Big[{d^2\over d\tau^2}\l\ln(\Lambda_{\alpha_\tau}+P_{0,\tau})\phi_{m,\tau},\phi_{m,\tau}\r\Big]_{|\tau=0}\nonumber\\
&=&
\sum_{p\in\Z}\big({3\over 2}+{1\over 2m}|p|\big) (m+p)^2|\hat G_{m+p}|^2+{1\over 2}m|\hat G_m|^2\nonumber\\
&&-{1\over 2}\sum_{p\in \Z}h(m,p)(m+|p|-\delta_p)^2(m+p)^2|\hat G_{m+p}|^2\nonumber\\
&&+{1\over 2}\sum_{p\in\Z\b\{0\}}(m-p)^2\ln(|p|)|\hat G_{m+p}|^2
-2\ln(m)\sum_{p\in\Z}\big({1\over 4}(m+p)^2+ m^2\big)|\hat G_{m+p}|^2\nonumber\\
&&+\sum_{p\in\Z}{m+|p|-\delta_p\over \rho(m,p,1)}(m-p)(m+p)|\hat G_{m+p}|^2.\label{K45}
\end{eqnarray}

The contribution over negative integers $p$ is 
\begin{eqnarray*}
&&\sum_{p<0}|\hat G_{m+p}|^2\Big[\big({3\over 2}-{1\over 2m}p\big) (m+p)^2-{1\over 2}h(m,p)(m-p)^2(m+p)^2\\
&&+{1\over 2}(m-p)^2\ln(|p|)
-{1\over 2}\ln(m)\big( 5m^2+2mp+p^2\big)
+{(m-p)^2(m+p)\over \rho(m,p,1)}\Big].
\end{eqnarray*}
We sustitute the values for $\rho$ and $h$ \eqref{K90a}, \eqref{K90b}. The contribution over the negative integers $p$ becomes
\begin{eqnarray}
&&\sum_{p<0}|\hat G_{m+p}|^2\Big[\big({3\over 2}-{1\over 2m}p\big) (m+p)^2-{1\over 2}(m-p)^2(\ln m-\ln|p|)\nonumber\\
&&+{1\over 2m}(m-p)^2(m+p)+{1\over 2}(m-p)^2\ln(|p|)
-{1\over 2}\ln(m)\big( 5m^2+2mp+p^2\big)\nonumber\\
&&+(m-p)^2(\ln m-\ln|p|)\Big]\nonumber\\
&=&2m(-\ln(m)+1)\sum_{p<0}(m+p)|\hat G_{m+p}|^2.\label{K47}
\end{eqnarray}
Then we note that 
$$
0={1\over 2}\int_0^{2\pi}(G^2)'=\int_0^{2\pi}fG
$$
which is written in Fourier series as 
\begin{equation}
m|\hat G_m|^2+\sum_{p<0}(m+p)|\hat G_{m+p}|^2+\sum_{p>0}(m+p)|\hat G_{m+p}|^2=0.\label{K52}
\end{equation}
Hence we combine \eqref{K47} and \eqref{K52} and the contribution in \eqref{K45} over the negative integers is given by 
\begin{equation}
2m^2(\ln(m)-1)|\hat G_m|^2+2m(\ln(m)-1)\sum_{p>0}(m+p)|\hat G_{m+p}|^2.\label{K48}
\end{equation}

Now let us look at the contribution when $p=0$ in \eqref{K45}. We split this case in two: When $m=1$ and when $m>1$.
First when $(m,p)=(1,0)$ the contribution in \eqref{K45} is 
\begin{equation}
|\hat G_1|^2({3\over 2}+{1\over 2})=2|\hat G_1|^2.\label{K49a}
\end{equation}
Next when $p=0$ and $m>1$ the contribution of $p$ in \eqref{K45} is given by 
\begin{eqnarray*}
|\hat G_m|^2\big[{3\over 2} m^2+{1\over 2}m
-{1\over 2}h(m,0)(m-1)^2m^2
-{5\over 2}\ln(m)m^2+m^2{m-1\over \rho(m,0,1)}\big].
\end{eqnarray*}
We sustitute the value for $h(m,0)={\ln(m)\over (m-1)^2}-{1\over m(m-1)}$ and $\rho(m,0,1)={m-1\over \ln(m)}$, see \eqref{K90a}--\eqref{K90b}, and the contribution becomes
\begin{eqnarray}
&&|\hat G_m|^2\big[{3\over 2} m^2+{1\over 2}m
-{1\over 2}\ln(m)m^2+{1\over 2}m(m-1)
-{5\over 2}\ln(m)m^2+m^2\ln(m)\big]\nonumber\\
&=&2m^2(1-\ln(m))|\hat G_m|^2.\label{K49b}
\end{eqnarray}
Then setting $m=1$ in \eqref{K49b} yields the same contribution as \eqref{K49a}.

Next we substitute \eqref{K48}, \eqref{K49b} into \eqref{K45} and we obtain
\begin{eqnarray}
&&\Big[{d^2\over d\tau^2}\l\ln(\Lambda_{\alpha_\tau}+P_{0,\tau})\phi_{m,\tau},\phi_{m,\tau}\r\Big]_{|\tau=0}\nonumber\\
&=&
2m(\ln(m)-1)\sum_{p=1}^{\infty}(m+p)|\hat G_{m+p}|^2\label{K50}\\
&&+\sum_{p=1}^\infty|\hat G_{m+p}|^2\big[({3\over 2}+{1\over 2m}p\big) (m+p)^2
-{1\over 2} h(m,p)(m+p)^4\nonumber\\
&&+{1\over 2}(m-p)^2\ln(p)
-2\ln(m)\big({1\over 4}(m+p)^2+ m^2\big)+{(m+p)^2\over \rho(m,p,1)}(m-p)\big].\nonumber
\end{eqnarray}
Let us look at the contribution of the $m$-th summand (when $p=m$): This is given by 
\begin{eqnarray*}
&&|\hat G_{2m}|^2\big(
4m^2(\ln(m)-1)
+({3\over 2}+{1\over 2}\big) 4m^2
-8 h(m,m)m^4
-4\ln(m)m^2\big)\\
&=&|\hat G_{2m}|^2\big(
4m^2(\ln(m)-1)
+({3\over 2}+{1\over 2}\big) 4m^2-4m^2-4\ln(m)m^2\big)=0.
\end{eqnarray*}
Therefore we rewrite \eqref{K50} as
\begin{eqnarray*}
&&\Big[{d^2\over d\tau^2}\l\ln(\Lambda_{\alpha_\tau}+P_{0,\tau})\phi_{m,\tau},\phi_{m,\tau}\r\Big]_{|\tau=0}\nonumber\\
&=&\sum_{p\in \N,\ p>0, p\not=m}|\hat G_{m+p}|^2\big[
2m(\ln(m)-1)(m+p)
+({3\over 2}+{1\over 2m}p) (m+p)^2\nonumber\\
&&-{1\over 2} h(m,p)(m+p)^4+{1\over 2}(m-p)^2\ln(p)\\
&&-2\ln(m)\big({1\over 4}(m+p)^2+ m^2\big)+{(m+p)^2\over \rho(m,p,1)}(m-p)\big].\nonumber
\end{eqnarray*}
Then we replace $\rho$ and $h$ by their definitions \eqref{K90a} and \eqref{K90b} and we obtain
\begin{eqnarray*}
&&\Big[{d^2\over d\tau^2}\l\ln(\Lambda_{\alpha_\tau}+P_{0,\tau})\phi_{m,\tau},\phi_{m,\tau}\r\Big]_{|\tau=0}\nonumber\\
&=&\sum_{p\in \N,\ p>0, p\not=m}|\hat G_{m+p}|^2\big[
2m(\ln(m)-1)(m+p)\nonumber\\
&&+({3\over 2}+{1\over 2m}p) (m+p)^2
-{1\over 2} ({\ln m-\ln p\over (m-p)^2}-{1\over m(m-p)})(m+p)^4\nonumber\\
&&+{1\over 2}(m-p)^2\ln(p)
-2\ln(m)\big({1\over 4}(m+p)^2+ m^2\big)+(\ln m-\ln p)(m+p)^2\big].
\end{eqnarray*}
We regroup the 2 lonely terms in $\ln(m)$ and we obtain
\begin{eqnarray*}
&&\Big[{d^2\over d\tau^2}\l\ln(\Lambda_{\alpha_\tau}+P_{0,\tau})\phi_{m,\tau},\phi_{m,\tau}\r\Big]_{|\tau=0}\nonumber\\
&=&\sum_{p\in \N,\ p>0, p\not=m}|\hat G_{m+p}|^2\big[
-2m(m+p)+({3\over 2}+{1\over 2m}p) (m+p)^2\nonumber\\
&&-{1\over 2} ({\ln m-\ln p\over (m-p)^2}-{1\over m(m-p)})(m+p)^4\nonumber\\
&&+{1\over 2}(m-p)^2(\ln(p)-\ln(m))+(\ln m-\ln p)(m+p)^2\big].
\end{eqnarray*}

Further elementary computations give \eqref{K60b}. In particular we regroup the terms in $\ln(p)-\ln(m)$ and we easily obtain that the coefficient in front of this difference is given by 
$$
{8p^2m^2\over (m-p)^2}|\hat G_{m+p}|^2={8p^2m^2\over (m-p)^2(m+p)^2}|\hat f_{m+p}|^2={8p^2m^2\over (m-p)^2(m+p)^2}|\hat \beta_{m+p}|^2.
$$
Regrouping the others terms yields
$$
4pm{m+p\over m-p}|\hat G_{m+p}|^2=4pm{(m+p)(m-p)\over (m-p)^2(m+p)^2}|\hat \beta_{m+p}|^2.
$$

\hfill $\Box$

\end{document}